\pdfoutput=1
\documentclass{amsart}
\usepackage{amsthm}
\usepackage{amsmath}
\usepackage{amsfonts}
\usepackage{amssymb}
\usepackage{mathrsfs}
\usepackage{tikz-cd}
\usepackage{caption}
\usepackage[utf8]{inputenc}
\usepackage{fullpage}
\usepackage{hyperref}
\usepackage{cleveref}

\theoremstyle{plain}
\newtheorem{thm}{Theorem}[section]
\newtheorem{lem}[thm]{Lemma}

\newtheorem{prop}[thm]{Proposition}
\newtheorem{cor}[thm]{Corollary}

\newtheorem{thmintro}{Theorem}

\newtheorem{corintro}[thmintro]{Corollary}

\theoremstyle{definition}
\newtheorem{defn}[thm]{Definition}

\newtheorem{rmk}[thm]{Remark}

\newtheorem*{ques*}{Question}
\newtheorem*{probl*}{Problem}

\newcommand{\mc}{\mathcal}
\newcommand{\mf}{\mathfrak}
\newcommand{\mscr}{\mathscr}
\newcommand{\sq}{\subseteq}
\newcommand{\ra}{\rightarrow}

\newcommand{\R}{\mathbb{R}}
\newcommand{\N}{\mathbb{N}}

\newcommand{\om}{\omega}
\newcommand{\Om}{\Omega}
\usepackage{xcolor}

\begin{document}

\title{On uniqueness of coarse median structures}

\author[E.\,Fioravanti]{Elia Fioravanti}\address{Institute of Algebra and Geometry, Karlsruhe Institute of Technology}\email{elia.fioravanti@kit.edu} 
\thanks{E.\,F.\ is supported by Emmy Noether grant 515507199 of the Deutsche Forschungsgemeinschaft (DFG)}

\author[A.\,Sisto]{Alessandro Sisto}
    \address{Department of Mathematics, Heriot-Watt University and Maxwell Institute for Mathematical Sciences, Edinburgh, UK}
    \email{a.sisto@hw.ac.uk}

\begin{abstract}
    We show that any product of bushy hyperbolic spaces has a unique coarse median structure, and that having a unique coarse median structure is a property closed under relative hyperbolicity. As a consequence, in contrast with the case of mapping class groups, there are non-hyperbolic pants graphs that have unique coarse median structures.
\end{abstract}

\maketitle

\section{Introduction}

Coarse median structures were introduced by Bowditch in \cite{Bow-coarsemedian} as a coarsening of the classical notion of median algebra. Roughly, a coarse median is a notion of centre for triangles in a given metric space. Bowditch's motivation was to study the coarse geometry of mapping class groups and Teichm\"uller spaces, in particular quasi-isometric rigidity \cite{Bow-largescale, Bow-Teich}. Other examples of coarse median groups include cocompactly cubulated groups and hierarchically hyperbolic groups \cite{Bow-largescale}.

Since a coarse median structure on a given metric space is an additional structure, it is natural to wonder how canonical it is. A coarse median on the metric space $X$ is in particular a map $\mu\colon X^3\to X$, and therefore it is natural to define two coarse medians to be \emph{equivalent} if they are bounded distance from each other. If all coarse medians on a given space are within bounded distance of each other, then we say that the space admits a \emph{unique coarse median structure}.

The problem of uniqueness of coarse median structures, besides being a natural question, is further motivated by the theory of coarse-median preserving automorphisms, which are better behaved than general automorphisms \cite{Fio10a,Fio10e,fioravanti2025growth}. For instance, fixed subgroups are finitely generated for coarse-median preserving automorphisms of cocompactly cubulated groups, but this is not true for general automorphisms. Clearly, if a group admits a unique coarse median structure, then all automorphisms preserve it.

The only groups that are so far known to have a unique coarse median structure are Gromov-hyperbolic groups \cite{NWZ1}, with some further examples having a coarse median structure that is unique among those coming from cocompact cubulations \cite{FLS}. Surprisingly to the authors, the mapping class group of the five-holed sphere does not have a unique coarse median structure \cite{Mangioni}, and in fact it has uncountably many. The same also holds for a lot of right-angled Artin groups, including the one on the pentagon (which however has a unique coarse median structure coming from cocompact cubulations).

The goal of this paper is to provide examples of groups and spaces that do have unique coarse median structures. First, we consider products of bushy hyperbolic spaces. A hyperbolic space $X$ is \emph{bushy} if there exists a constant $\lambda\geq 0$ such that, for all points $x\in X$, there exist three ideal points $\xi_1,\xi_2,\xi_3\in\partial_{\infty}X$ with pairwise Gromov products based at $x$ bounded above by $\lambda$.

\begin{thmintro}
\label{thmintro:A}
    Let $X$ be a finite product of bushy geodesic hyperbolic spaces. Then $X$ admits a unique coarse median structure.
\end{thmintro}

Note that we do need the hyperbolic spaces to be bushy (Remark \ref{rmk:R_factor}), and in fact this is the phenomenon exploited in \cite{Mangioni} to construct different coarse median structures on a mapping class group.

Next, we show that having a unique coarse median structure is a property closed under relative hyperbolicity. Existence of coarse median structures in this case is due to Bowditch \cite{Bow-cmrelhyp}; in this paper we deal with uniqueness only.

\begin{thmintro}
\label{thmintro:B}
    Let $(X,\mc{P})$ be a geodesic, relatively hyperbolic space. Suppose that each $P\in\mc{P}$ admits a unique coarse median structure, and that there are only finitely many isometry types in $\mc{P}$. Then $X$ admits a unique coarse median structure.
\end{thmintro}

The two theorems imply for instance that the right-angled Artin groups $F_2\times F_2$ and $\mathbb Z* (F_2\times F_2)$ admit unique coarse median structures.

Combining the two results we also obtain uniqueness of coarse median structures for pants graphs in low complexity, in contrast with the mapping class group case. The pants graph of a finite-type surface was introduced by Hatcher and Thurston \cite{Hatcher-Thurston}. Brock showed that it is quasi-isometric to Teichm\"uller space with the Weil--Petersson metric \cite{Brock}.

\begin{corintro}
    Let $S$ be a connected finite-type surface of genus $g$ with $p$ punctures, and suppose that $3g+p\leq 6$. Then the pants graph of $S$ has a unique coarse median structure.
\end{corintro}

\begin{proof}
    Given the constraints on $g$ and $p$, the pants graph is (either hyperbolic or) hyperbolic relative to products of two Farey graphs, see \cite[Theorem 1.1]{BF:hyp} and \cite[Theorem 1]{BM:rel_hyp}. Theorems \ref{thmintro:A} and \ref{thmintro:B} therefore apply.
\end{proof}

Mangioni's arguments suggest that mapping class groups of higher-complexity surfaces should have uncountably many coarse median structures, but there is a chance that the answer to the following is affirmative:

\begin{ques*}
    Do all pants graphs of finite-type surfaces have a unique coarse median structure?
\end{ques*}

The main challenge in answering the question above is that we have much less control on quasigeodesics in pants graphs that are not relatively hyperbolic.

There are several other natural problems that arise regarding the classification of groups with unique coarse median structures, including of particular interest:

\begin{probl*}
    Classify right-angled Artin groups with unique coarse median structures.
\end{probl*}

\subsection*{Outline and proof ideas}

In Section \ref{sec:prelim} we cover preliminaries, the only novel result being Corollary \ref{cor:Fuerstenberg}, which gives a criterion for a quasiflat in a coarse median space to be quasiconvex.

In Section \ref{sec:rel_hyp}, we prove Theorem \ref{thmintro:B}(=\ref{thm:rel_hyp}). Here we exploit the fact that intervals in coarse median spaces contain uniform quasigeodesics (Lemmas \ref{lem:quasi-geodesics_in_intervals} and \ref{lem:quasi-geodesics_through_point}) and that quasigeodesic triangles in relatively hyperbolic spaces are well-understood.

In Section \ref{sec:prod_hyp}, we prove Theorem \ref{thmintro:A}(=\ref{thm:prod_hyp_unique_median}). We exploit uniqueness of median operators in products of furry trees (which follows from work of Bowditch) and tricks with asymptotic cones to show that flats in products of bushy hyperbolic spaces are quasiconvex with respect to any coarse median structure. From this it is not hard to deduce uniqueness of the coarse median structure. (Uniqueness of the median operators on asymptotic cones does not seem to directly pull back to uniqueness of the coarse median structure on the space, in general.)

\section{Preliminaries}\label{sec:prelim}
\subsection{Median algebras and median spaces}

In this subsection we recall basic notions related to median algebras, see e.g.\ \cite{CDH,Bow-cmrelhyp,Bow-book} for an introduction.

Let $\Om$ be a set and $m\colon\Om^3\ra\Om$ a map. The pair $(\Om,m)$ is a \emph{median algebra} if the following identities are satisfied for all elements $x,y,z,p\in\Om$:
\begin{enumerate}
\item[(M0)] symmetry and localisation: $m(x,y,z)=m(y,x,z)=m(x,z,y)$ and $m(x,x,y)=x$;
\item[(M1)] associativity: $m(m(x,p,y),p,z)=m(x,p,m(y,p,z))$.
\end{enumerate}
We refer to any map $m$ satisfying these identities as a \emph{median operator} on the set $\Om$. 

The \emph{$n$--dimensional discrete cube}, for some $n\geq 0$, is the finite median algebra consisting of the set $\{0,1\}^n$ with the median operator computing majority votes on coordinates. The \emph{rank} of a median algebra $(\Om,m)$ is the largest integer $r$ such that $\Om$ contains an $r$--dimensional discrete cube as a median subalgebra; if such a largest integer does not exist, then the rank is infinite. We also define an \emph{$n$--dimensional Euclidean cube} to be a product $\prod_{1\leq i\leq n}[a_i,b_i]$ of nontrivial compact intervals $[a_i,b_i]\sq\R$, equipped with the median operator that chooses middle coordinates.

An \emph{interval} in a median algebra $(\Om,m)$ is a set of the form $[x,y]=\{m(x,y,z)\mid z\in\Om\}$ for some $x,y\in\Om$. A subset $A\sq \Om$ is \emph{convex} if we have $[a_1,a_2]\sq A$ for all $a_1,a_2\in A$. Intervals are convex.

A \emph{topological median algebra} is the data of a median algebra $(\Om,m)$ and a Hausdorff topology on $\Om$ with respect to which the map $m$ is continuous. If $\Om$ is a metric space, we say that the median operator $m$ is \emph{Lipschitz} if there exists a constant $L\geq 1$ such that the map $x\mapsto m(x,y,z)$ is $L$--Lipschitz for all $y,z\in\Om$.

A \emph{median space} is a metric space $(\Om,\rho)$ with the property that, for any three points $x_1,x_2,x_3\in\Om$, there exists a unique point $m=m(x_1,x_2,x_3)\in\Om$ such that 
\[\rho(x_i,x_j)=\rho(x_i,m)+\rho(m,x_j)\] 
for all indices $1\leq i<j\leq 3$. If $(\Om,\rho)$ is a median space and $m\colon\Om^3\ra\Om$ is the resulting ternary map, then the pair $(\Om,m)$ is a median algebra; we refer to $m$ as the \emph{induced} median operator. Conversely, if $(\Om,m)$ is a median algebra and $\rho$ is a metric on $\Om$ with the property that $\rho(x,y)=\rho(x,m(x,y,z))+\rho(m(x,y,z),y)$ for all $x,y,z\in\Om$, then the pair $(\Om,\rho)$ is a median space.

The following result is due to Zeidler, see \cite[Proposition~3.3 and Lemma~3.5]{Zeidler}.

\begin{prop}\label{prop:Zeidler}
Let $(\Om,d)$ be a geodesic metric space, and let $m\colon\Om^3\ra\Om$ be a Lipschitz median operator with finite rank. Then $\Om$ admits a metric $\rho$ bi-Lipschitz to $d$ such that $(\Om,\rho)$ is a median space inducing the median operator $m$.
\end{prop}

\subsection{Coarse median spaces}
In this section we recall the definition of coarse median space and related notions, and establish basic properties that will be useful later on.

Let $X$ be a metric space. For $x,y\in X$ and $D\geq 0$, we write ``$x\approx_Dy$'' with the meaning of ``$d(x,y)\leq D$''. A \emph{$C$--coarse median} on $X$, for some constant $C\geq 1$, is a map $\mu\colon X^3\ra X$ satisfying the following identities for all points $x,y,z,p\in X$:
\begin{enumerate}
\item[(M0)] symmetry and localisation: $\mu(x,y,z)=\mu(y,x,z)=\mu(x,z,y)$ and $\mu(x,x,y)=x$;
\item[(CM1)] coarse associativity: $\mu(\mu(x,p,y),p,z)\approx_C\mu(x,p,\mu(y,p,z))$;
\item[(CM2)] coarse Lipschitz property: $d(\mu(x,y,z),\mu(p,y,z))\leq Cd(x,p)+C$.
\end{enumerate}
We say that $\mu$ is a \emph{coarse median} if it is a $C$--coarse median for some $C\geq 1$. The pair $(X,\mu)$ is then a \emph{coarse median space}, or a $C$--coarse median space if we wish to specify the constant. (See \cite{NWZ1} for the equivalence between the above definition and Bowditch's original one in \cite{Bow-coarsemedian}.)

Two coarse medians $\mu,\nu\colon X^3\ra X$ are \emph{$D$--close}, for a constant $D\geq 0$, if we have $\mu(x,y,z)\approx_D\nu(x,y,z)$ for all points $x,y,z\in X$. We simply say that $\mu$ and $\nu$ are \emph{close} if we do not wish to specify the constant, and note that closeness is an equivalence relation. We refer to each equivalence class of coarse medians on $X$, denoted $[\mu]$, as a \emph{coarse median structure}.

There is a notion of \emph{rank} for coarse median spaces extending the analogous notion for median algebras. Roughly, $(X,\mu)$ has rank $\leq r$ if all its finite subsets can be approximated by subsets of median algebras with rank $\leq r$. The exact definition is a bit more involved and will never be needed explicitly, so we refer to \cite[Section~8]{Bow-coarsemedian}.

In a coarse median space $(X,\mu)$, the \emph{interval} between two points $x,y\in X$ is the set
\[ [x,y]_{\mu}:=\{\mu(x,y,z) \mid z\in X \} . \]
We simply denote this by $[x,y]$ when there is no ambiguity on the coarse median under consideration. There is a natural map $\pi_{x,y}\colon X\ra[x,y]$ given by $\pi_{x,y}(z)=\mu(x,y,z)$, and we refer to it as the \emph{projection} to the interval. A subset $A\sq X$ is \emph{$D$--quasiconvex}, for some $D\geq 0$, if the interval $[a_1,a_2]$ is contained in the $D$--neighbourhood of $A$ for all points $a_1,a_2\in A$. Intervals are uniformly quasiconvex, see \Cref{lem:cm_basics}(4).

\begin{rmk}
Whenever we use the term \emph{quasiconvex} in the sequel, this will always be meant with respect to a fixed coarse median structure. We will never consider quasiconvexity in a metric sense, which would amount to requiring geodesics (or quasigeodesics) with endpoints in the set to remain close to it. Note that the only general connection between coarse median intervals and quasigeodesics is the rather weak one provided by \Cref{lem:quasi-geodesics_in_intervals} below. We refer to \cite[Theorem~5.1]{NWZ1} for an example where geodesics stray arbitrarily far from the interval between their endpoints.
\end{rmk}

For a subset $A\sq X$ and $R\geq 0$, we denote by $\mc{N}_R(A)$ the (closed) $R$--neighbourhood of $A$ in $X$. The following lemma collects various basic facts and identities that will be useful in the sequel.

\begin{lem}\label{lem:cm_basics}
Let $(X,\mu)$ be a $C$--coarse median space. There is $K=K(C)\geq 1$ such that the following hold.
\begin{enumerate}
\item For all $a,b,c,d,e\in X$, we have $\mu(\mu(a,b,c),d,e)\approx_K\mu(\mu(a,d,e),\mu(b,d,e),c)$.
\item Given $x,y\in X$, $p\in [x,y]$ and points $x'\in[x,p]$, $y'\in[p,y]$, we have $\mu(x',y',p)\approx_K p$.
\item For any $x,y,z\in X$, we have $\mu(x,y,z)\in [x,y]\cap [y,z]\cap [z,x]$. For any $R\geq 0$, the intersection 
\[ \mc{N}_R([x,y])\cap\mc{N}_R([y,z])\cap\mc{N}_R([z,x]) \] 
has diameter at most $KR+K$.
\item If $x,y\in X$ and $a,b\in\mc{N}_R([x,y])$ for some $R\geq 0$, then the interval $[a,b]$ is contained in the neighbourhood of $[x,y]$ of radius  $KR+K$.
\item For $x,y,z\in X$ and $R\geq 0$, consider three points $a\in\mc{N}_R([x,y])\cap\mc{N}_R([x,z])$, $b\in\mc{N}_R([y,z])\cap\mc{N}_R([y,x])$ and $c\in\mc{N}_R([z,x])\cap\mc{N}_R([z,y])$. Then $\mu(a,b,c)$ and $\mu(x,y,z)$ are at distance $\leq KR+K$.
\item If $A\sq X$ is a $D$--quasiconvex subset and $B\sq X$ is a set with $d_{\rm Haus}(A,B)\leq D$, for some $D\geq 0$, then $B$ is $2(CD+C+D)$--quasiconvex.
\end{enumerate}
\end{lem}
\begin{proof}
Part~(1) follows from the main theorem of \cite{NWZ1}, see e.g.\ \cite[Section~2.3]{NWZ2}. Parts~(2) and~(3) were shown in \cite{NWZ2}: they follow, respectively, from a double application of Lemma~3.6 and from Proposition~3.2(I3) in that article.

We now prove part~(4). Suppose first that $R=0$. In this case we have $a,b\in[x,y]$, and hence $\mu(x,y,a)\approx_Ca$ and $\mu(x,y,b)\approx_Cb$, by identity (CM1). Given $p\in[a,b]$, we can write $p=\mu(a,b,p')$ for some $p'\in X$ and, using part~(1) and identity (CM2), we obtain
\[ \mu(\mu(a,b,p'),x,y)\approx_K\mu(\mu(a,x,y),\mu(b,x,y),p')\approx_{2C^2+2C} \mu(a,b,p')=p. \]
This shows that $[a,b]$ is contained in the neighbourhood of $[x,y]$ of radius $K+2C^2+2C$, as required. For the general case, it suffices to consider points $a',b'\in[x,y]$ at distance $\leq R$ from $a$ and $b$, respectively. The interval $[a,b]$ is then contained in the $2(CR+R)$--neighbourhood of $[a',b']$, by a double application of identity (CM2), and $[a',b']$ is contained in the $(K+2C^2+2C)$--neighbourhood of $[x,y]$ by the previous discussion.

Regarding part~(5), the point $\mu(a,b,c)$ lies in the intersection $[a,b]\cap[b,c]\cap[c,a]$. By part~(4), the latter is contained in $\mc{N}_{R'}([x,y])\cap\mc{N}_{R'}([y,z])\cap\mc{N}_{R'}([z,x])$ for $R'=KR+K$. By part~(3), the latter intersection has diameter $\leq KR'+K$ and also contains the point $\mu(x,y,z)$. In conclusion, the distance of $\mu(a,b,c)$ and $\mu(x,y,z)$ is at most $K^2R+K^2+K$, as desired.

Finally, we address part~(6). Consider two points $b_1,b_2\in B$ and let $a_1,a_2\in A$ be points with $d(a_i,b_i)\leq D$. By identity~(CM2), the interval $[b_1,b_2]$ is contained in the $2(CD+C)$--neighbourhood of the interval $[a_1,a_2]$. The latter is contained in $\mc{N}_D(A)\sq\mc{N}_{2D}(B)$, hence we obtain the desired conclusion.
\end{proof}

\subsection{Ultralimits}\label{sub:ultralimits}

Let $\om$ be a non-principal ultrafilter. Given a sequence of metric spaces $(X_n,d_n)$ and basepoints $p_n\in X$, we can consider the set of all sequences $x_n\in X_n$ such that $\lim_{\om} d_n(x_n,p_n)<+\infty$, and endow this set with the pseudo-metric 
\[ d_{\om}\big((x_n),(x_n')\big):=\lim_{\om}d_n(x_n,x_n') .\]
We denote by $(X_{\om},d_{\om})$ the quotient metric space of this pseudo-metric space, which is known as the \emph{$\om$--limit} (or \emph{ultralimit}) of the sequence of pointed metric spaces $(X_n,p_n)$. We think of a sequence of points $x_n\in X_n$ as converging to the point $(x_n)\in X_{\om}$. We refer to \cite[Chapter~10]{DrutuKapovich} for generalities on ultralimits and ultrafilters.

To any sequence of subsets $A_n\sq X_n$, we can associate the set $A_{\om}\sq X_{\om}$ formed by all points of the form $(a_n)$ with $a_n\in A_n$ for all $n$ (or, equivalently, only those with $a_n\in A_n$ for $\om$--all $n$). Following \cite[Section~10]{Bow-quasiflats}, we say that $A_{\om}$ is the \emph{strong limit} of the $A_n$ and write $A_n\ra A_{\om}$. If $\mscr{A}_n$ are families of subsets of $X_n$, we also say that $A_{\om}$ is a \emph{strong limit} of the $\mscr{A}_n$ if there exist sets $A_n\in\mscr{A}_n$ with $A_n\ra A_{\om}$. Finally, a set $A\sq X_{\om}$ is a \emph{weak limit} of the $\mscr{A}_n$ if each of its bounded subsets is contained in a strong limit of the $\mscr{A}_n$.

If $\mu_n\colon X_n^3\ra X_n$ is a $C_n$--coarse median and $C_{\om}:=\lim_{\om}C_n<+\infty$, then we obtain a $C_{\om}$--coarse median $\mu_{\om}\colon X_{\om}^3\ra X_{\om}$, defined (with an abuse) by the formula
\[ \mu_{\om}\big((x_n),(y_n),(z_n)\big)=\big(\mu_n(x_n,y_n,z_n)\big) .\]
To be precise, one should choose specific sequences $x_n',y_n',z_n'$ representing the points $(x_n),(y_n),(z_n)$ and define the coarse median as the point $\big(\mu_n(x_n',y_n',z_n')\big)$. Choosing different sequences $x_n',y_n',z_n'$ can alter the definition of the map $\mu_{\om}$, because of the additive error in identity~(CM2). However, independently of the choices, we obtain a $C_{\om}$--coarse median on $X_{\om}$, and different choices will yield a $C_{\om}$--coarse median that is $3C_{\om}$--close to it. Thus, the corresponding coarse median \emph{structure} on $X_{\om}$ is unaffected by all choices.

We have a particular kind of ultralimit when the metric spaces $(X_n,d_n)$ are rescaled copies of a fixed metric space $(X,d)$, namely $(X,\lambda_nd)$ for a sequence $\lambda_n\ra 0$. In this case, the ultralimit is known as an \emph{asymptotic cone} of $X$. Note that asymptotic cones still depend on the choice of basepoints and ultrafilter, in general. It is important to observe that asymptotic cones of coarse median spaces actually inherit a uniquely-defined structure of median algebra, with a Lipschitz median operator. Indeed, since $\lambda_n\ra 0$, the additive errors in identities (CM1) and (CM2) disappear in the limit.

Ultralimits are well-behaved with respect to the notion of rank. An ultralimit of coarse median spaces of rank $\leq r$ is again a coarse median space of rank $\leq r$. All asymptotic cones of a coarse median space of rank $\leq r$ are median algebras of rank $\leq r$, see \cite[Proposition~9.3]{Bow-coarsemedian}.

We will also need a result of Bowditch that can be used to compare two families of subsets of a coarse median space, based on the behaviour of their strong limits in all asymptotic cones (\Cref{lem:Bowditch} below). In order to state this, we need to introduce some of Bowditch's notation. (In the original application, roughly, $\mc{E}$ is the collection of all quasiflats in $X$, $\mc{F}$ are certain coarsely cubulated subspaces, $\mc{E}^{\infty}(X^{\infty})$ are bilipschitz flats in the asymptotic cone, and $\mc{F}^{\infty}(X^{\infty})$ are certain cubulated subspaces.)

Let $(X,\mu)$ be a coarse median space and let $\mc{E},\mc{F}$ be two families of subsets of $X$. For each asymptotic cone $X^{\infty}$ of $X$, let $\mc{E}^{\infty}(X^{\infty})$ and $\mc{F}^{\infty}(X^{\infty})$ be two families of subsets of $X^{\infty}$. We assume that the following conditions are satisfied for all asymptotic cones $X^{\infty}$ of $X$ (these are a slightly simplified version of Bowditch's conditions in \cite[Section~10]{Bow-quasiflats}, which suffices for our purposes).
\begin{enumerate}
\item[(E1)] All strong limits of $\mc{E}$ lie in $\mc{E}^{\infty}(X^{\infty})$, and all strong limits of $\mc{F}$ lie in $\mc{F}^{\infty}(X^{\infty})$. All elements of $\mc{F}^{\infty}(X^{\infty})$ are weak limits of $\mc{F}$.
\item[(E2)] Each element of $\mc{E}^{\infty}(X^{\infty})$ is a subset of an element of $\mc{F}^{\infty}(X^{\infty})$.
\item[(E3)] If some $E\in\mc{E}^{\infty}(X^{\infty})$ is contained in a bounded neighbourhood of some $F\in\mc{F}^{\infty}(X^{\infty})$, then $E\sq F$.
\end{enumerate}
We will use the following result.

\begin{lem}[\cite{Bow-quasiflats}, Lemma~10.5]
\label{lem:Bowditch}
For every $C\geq 1$, there exists a constant $R=R(C)$ such that the following holds. 
Let $(X,\mu)$ be a $C$--coarse median space, with $X$ geodesic. 
Let $\mc{E},\mc{F},\mc{E}^{\infty}(\cdot),\mc{F}^{\infty}(\cdot)$ be families of subsets satisfying conditions (E1)--(E3) above. Then each bounded subset of each element of $\mc{E}$ is contained in the $R$--neighbourhood of an element of $\mc{F}$.
\end{lem}

\subsection{Quasiflats in coarse median spaces}

In this subsection we establish a criterion for a quasiflat in a coarse median space to be quasiconvex.

For simplicity, given a constant $C\geq 1$, we refer to $(C,C)$--quasi-isometric embeddings and $(C,C)$--quasi-isometries simply as $C$--quasi-isometric embeddings and $C$--quasi-isometries. We will need the following classical result; see e.g.\ \cite[Corollary~2.6]{KapovichLeeb97} for a proof. 

\begin{prop}\label{prop:Fuerstenberg}
There exists a constant $D=D(C,p)$ such that, for any $C$--quasi-isometric embedding $f\colon\R^p\ra\R^p$, the image of $f$ is $D$--dense in the codomain.
\end{prop}

We will use the proposition in the following form. (Recall that quasiconvexity is always meant in the coarse median sense, and not in a metric sense.) 

\begin{cor}\label{cor:Fuerstenberg}
For all $C\geq 0$ and $p\in\N$, there exists a constant $C_1=C_1(C,p)$ such that the following holds. Let $(X,\mu)$ be a $C$--coarse median space and let $E\sq X$ be a subset $C$--quasi-isometric to $\R^p$. Suppose that, for every bounded subset $B\sq E$, we have $B\sq\mc{N}_C(Y_B)$ for a $C$--quasiconvex subset $Y_B\sq X$ such that $Y_B$ embeds $C$--quasi-isometrically in $\R^p$. Then $E$ is $C_1$--quasiconvex.
\end{cor}
\begin{proof}
To begin with, suppose that we actually have a single subset $Y\sq X$ such that $E\sq\mc{N}_C(Y)$ and such that $Y$ is $C$--quasiconvex and admits a $C$--quasi-isometric embedding $f\colon Y\ra\R^p$. Mapping each point of $E$ to a $C$--close point of $Y$ and then composing with $f$, we obtain a quasi-isometric embedding $E\ra\R^p$. Since $E$ is itself $C$--quasi-isometric to $\R^p$, \Cref{prop:Fuerstenberg} implies that $E$ and $Y$ are at Hausdorff distance at most $C'$, for a constant $C'$ depending only on $C$ and $p$. Finally, the fact that $Y$ is $C$--quasiconvex implies that $E$ is $C''$--quasiconvex for some constant $C''=C''(C,C')$, using \Cref{lem:cm_basics}(6). 

Now, we reduce the general case to the previous one by taking an ultralimit (without rescaling!). For this, let $C''$ be the constant obtained in the previous paragraph, and suppose for the sake of contradiction that that there exist $X,\mu,E$, as in the statement of the corollary, for which $E$ is not $(C''+1)$--quasiconvex. Thus, there are points $x,y\in E$ and $z\in [x,y]$ such that $d(z,E)\geq C''+1$. For each $n\in\N$, there exists a $C$--quasiconvex subset $Y_n\sq X$ such that $Y_n$ admits a $C$--quasi-isometric embedding in $\R^p$ and such that the neighbourhood $\mc{N}_C(Y_n)$ contains the intersection $E\cap\mc{N}_n(x)$. Now, choose a non-principal ultrafilter $\om$ and define $X_{\om}$ as the $\om$--limit of countably many copies of $X$, each based at the point $x$. Let $E_{\om}\sq X_{\om}$ and $Y_{\om}\sq X_{\om}$ be the subspaces obtained as the strong limits, respectively, of the constant sequence $E$ and of the sequence $Y_n$. Also let $\mu_{\om}$ be the $\om$--limit of the coarse medians $\mu_n$, as defined above. We have that $\mu_{\om}$ is again a $C$--coarse median, that $E_{\om}$ is again $C$--quasi-isometric to $\R^p$, and that $Y_{\om}$ is again $C$--quasiconvex and $C$--quasi-isometrically embeddable in $\R^p$. In addition, since the neighbourhoods $\mc{N}_C(Y_n)$ contain larger and larger balls in the set $E$, we have gained in the ultralimit that $E_{\om}\sq\mc{N}_C(Y_{\om})$. We are thus in the situation discussed at the beginning of the proof, so $E_{\om}$ must be $C''$--quasiconvex. At the same time, the constant sequences $x_n=x,y_n=y,z_n=z$ converge to points $x_{\om},y_{\om}\in E_{\om}$ and $z_{\om}\in[x_{\om},y_{\om}]$ with $d(z_{\om},E_{\om})\geq C''+1$, a contradiction.
\end{proof}

\section{Coarse medians and (relative) hyperbolicity}\label{sec:rel_hyp}

\subsection{Hyperbolic spaces and the Morse property}

In this subsection, we give a short proof of the fact that Gromov-hyperbolic metric spaces admit a unique coarse median operator, a fact originally shown in \cite[Theorem~4.2]{NWZ1}. We then go on to deduce that Morse subsets of metric spaces are quasiconvex with respect to all coarse median operators. Recall that a subset of a metric space $A\sq X$ is said to be \emph{$N$--Morse}, for a function $N\colon\R_{\geq 1}\ra\R_{\geq 0}$, if every $(L,L)$--quasigeodesic in $X$ with endpoints on $A$ is contained in the $N(L)$--neighbourhood of $A$.

For both results, the key observation is the following.

\begin{lem}\label{lem:quasi-geodesics_in_intervals}
Let $(X,\mu)$ be a $C$--coarse median space, with $X$ a geodesic metric space. For any $x,y\in X$, the interval $[x,y]$ contains a $(2C,4C)$--quasigeodesic from $x$ to $y$.
\end{lem}
\begin{proof}
We will construct a sequence of points $u_0,\dots,u_m\in [x,y]$ with $u_0=x$, $u_m=y$, $d(u_i,u_{i+1})\leq 2C$ and $d(u_i,u_j)\geq |i-j|-3$ for all $0\leq i<j\leq m$. In order to get the required quasigeodesic $\gamma\colon [0,m]\ra[x,y]$ from $x$ to $y$, one can then simply define $\gamma(t):=u_{\lfloor t \rfloor}$.

Thus, let $u_0:=x,u_1,\dots,u_m:=y$ be a sequence of points in $[x,y]$ such that $d(u_i,u_{i+1})\leq 2C$ for all $i$, and such that $m$ is the smallest possible. (The coming argument also shows that such a sequence exists.) For some indices $i,j$, let $\alpha\sq X$ be a geodesic from $u_i$ to $u_j$ and consider, for each integer $0\leq k\leq d(u_i,u_j)$, the point $y_k\in\alpha$ with $d(u_i,y_k)=k$. Let $y_k':=\mu(x,y,y_k)$ be the projection of $y_k$ to the interval $[x,y]$. By identity (CM2), we have $d(y_k',y_{k+1}')\leq 2C$ for all $k$. Moreover, by identity (CM1), the point $y_0'=\mu(x,y,u_i)$ is at distance at most $C$ from $u_i$, as we have $u_i\in[x,y]$. Similarly, the point $\mu(x,y,u_j)$ is within distance $C$ of $u_j$, and within distance $2C$ of the last of the $y_k'$. 

In conclusion, the points $y_k'$ (together with the point $\mu(x,y,u_j)$) give a sequence of at most $d(u_i,u_j)+3$ steps of length at most $2C$ leading from $u_i$ and $u_j$ within the interval $[x,y]$. By minimality of the integer $m$ in our choice of the points $u_i$, it follows that $|i-j|\leq d(u_i,u_j)+3$, as desired.
\end{proof}

\begin{cor}
\label{cor:unique_hyp}
    If $X$ is a geodesic hyperbolic space, then $X$ has a unique coarse median structure. More precisely, if $X$ is $\delta$--hyperbolic, there is a constant $D_1=D_1(\delta,C)$ such that any two $C$--coarse medians on $X$ are $D$--close to each other.
\end{cor}
\begin{proof}
By hyperbolicity, there exist a constant $D'=D'(\delta,C)$ and a map $m\colon X^3\ra X$ such that, for any three points $x,y,z\in X$ and any $(2C,4C)$--quasigeodesic triangle $T$ with these vertices, the point $m(x,y,z)$ is at distance at most $D'$ from each of the three sides of $T$. If $\mu$ is any $C$--coarse median on $X$, then each of the three intervals $[x,y]_{\mu},[y,z]_{\mu},[z,x]_{\mu}$ contains a $(2C,4C)$--quasigeodesic joining its endpoints, by \Cref{lem:quasi-geodesics_in_intervals}. It follows that the point $m(x,y,z)$ is at distance at most $D'$ from each of the intervals $[x,y]_{\mu},[y,z]_{\mu},[z,x]_{\mu}$. Now, \Cref{lem:cm_basics}(3) implies that the points $m(x,y,z)$ and $\mu(x,y,z)$ are at distance at most $\tfrac{1}{2}D:=KD'+K$, for a constant $K$ depending only on $C$. Thus, any two $C$--coarse medians on $X$ are $D$--close.
\end{proof}

In order to handle Morse subsets, we will require the following strengthening of \Cref{lem:quasi-geodesics_in_intervals}.

\begin{lem}\label{lem:quasi-geodesics_through_point}
Let $(X,\mu)$ be a $C$--coarse median space, with $X$ a geodesic metric space.
\begin{enumerate}
\item Consider points $x,y\in X$ and $p\in [x,y]$. For some $L\geq 1$, let $\alpha\sq [x,p]$ and $\beta\sq[p,y]$ be $(L,L)$--quasigeodesics, respectively, from $x$ to $p$ and from $p$ to $y$. Then, the union $\alpha\cup\beta$ is a $(L',L')$--quasigeodesic from $x$ to $y$, for some $L'=L'(C,L)$.
\item For all $x,y\in X$, each point $p\in[x,y]$ lies on a $(C_2,C_2)$--quasigeodesic from $x$ to $y$ contained in $[x,y]$, for some constant $C_2=C_2(C)$.
\end{enumerate}
\end{lem}
\begin{proof}
We begin with the proof of part~(1). Let $K=K(C)$ be the constant provided by \Cref{lem:quasi-geodesics_in_intervals}. For any $x'\in\alpha$ and $y'\in\beta$, we have $\mu(x',p,y')\approx_K p$ by \Cref{lem:quasi-geodesics_in_intervals}(2). It then follows that
\[ d(x',p) \approx_K d\big(x',\mu(x',p,y')\big) = d\big(\mu(x',p,x'),\mu(x',p,y')\big) \leq Cd(x',y')+C, \]
where the last inequality is given by identity (CM2), after using symmetry of $\mu$. As an analogous chain of inequalities applies to the distance $d(p,y')$, we conclude that 
\[ d(x',p)+d(p,y') \leq 2Cd(x',y')+2(C+K) .\]
Now, setting $\eta:=\alpha\cup\beta$ and letting $t,s$ be the times at which $\eta(t)=x'$ and $\eta(s)=y'$, and using the fact that $\alpha$ and $\beta$ are $(L,L)$--quasigeodesics, we obtain
\[ |t-s| \leq Ld(x',p)+Ld(p,y')+2L^2 \leq 2CLd(x',y')+2L(C+K+L) .\]
Since the points $x'\in\alpha$ and $y'\in\beta$ were arbitrary, this shows that $\eta$ is a quasigeodesic with parameters depending only on $L$ and $C$, proving part~(1).

Regarding part~(2), consider $x,y\in X$ and a point $p\in [x,y]$. By \Cref{lem:quasi-geodesics_in_intervals}, there exist $(2C,4C)$--quasigeodesics $\alpha\sq[x,p]$ and $\beta\sq[p,y]$ connecting the endpoints of these two intervals. By part~(1), the union $\eta:=\alpha\cup\beta$ is a quasigeodesic from $x$ to $y$, with quality only depending on $C$. The quasigeodesic $\eta$ is contained in the union $[x,p]\cup[p,y]$, which is contained in the $K$--neighbourhood of the interval $[x,y]$ by \Cref{lem:cm_basics}(4). Projecting $\eta$ onto $[x,y]$ yields the required quasigeodesic from $x$ to $y$.  
\end{proof}

 Thus, part~(2) of \Cref{lem:quasi-geodesics_through_point} immediately implies the following:

\begin{cor}\label{cor:Morse_qc}
Let $(X,\mu)$ be a $C$--coarse median space, with $X$ a geodesic metric space. Let $A\sq X$ be an $N$--Morse subset. Then $A$ is $N(C_2)$--quasiconvex, for the constant $C_2=C_2(C)$ from \Cref{lem:quasi-geodesics_through_point}.
\end{cor}

\subsection{Relatively hyperbolic spaces}

Bowditch showed that the existence of a coarse median structure is invariant under relative hyperbolicity. That is, if a geodesic metric space $X$ is hyperbolic relative to a collection of subspaces $\mc{P}$, and if there exists a constant $C\geq 1$ such that each $P\in\mc{P}$ admits a $C$--coarse median, then $X$ itself admits a coarse median \cite[Theorem~2.1]{Bow-cmrelhyp}.

Here, we address the question of the uniqueness of such coarse medians, extending \Cref{cor:unique_hyp} above.

\begin{thm}
\label{thm:rel_hyp}
Let a geodesic metric space $X$ be hyperbolic relative to a collection of subspaces $\mc{P}$. Suppose that for all constants $C\geq 1$ there exists a constant $D\geq 1$ such that, for all $P\in\mathcal P$, all $C$--coarse medians on $P$ are $D$--close to each other. Then $X$ admits at most one coarse median structure.
\end{thm}
\begin{proof}
Let us suppose that $\mu$ and $\nu$ are two $C$--coarse medians on $X$, for some $C\geq 1$, and show that $\mu$ and $\nu$ must be close to each other. Each subspace $P\in\mc{P}$ is $N$--Morse in $X$, for a fixed function $N$, by \cite[Lemma~4.15]{DrutuSapir05}. Appealing to \Cref{cor:Morse_qc}, it follows that there exists a constant $D'\geq 0$ such that all $P\in\mc{P}$ are $D'$--quasiconvex with respect to both $\mu$ and $\nu$. In particular, the restrictions of $\mu$ and $\nu$ to each $P\in\mc{P}$ are coarse medians of bounded quality and, by our hypothesis, there exists a constant $D''\geq 0$ such that $\mu(p_1,p_2,p_3)\approx_{D''} \nu(p_1,p_2,p_3)$ for all $P\in\mc{P}$ and all $p_1,p_2,p_3\in P$.

Now, consider three arbitrary points $x,y,z\in X$. By \cite[Lemma~8.19]{DrutuSapir05}, these three points have a ``barycentre'', which is either a point $b\in X$ or a subspace $B\in\mc{P}$. This barycentre is characterised by the fact that there exists a constant $\Delta\geq 0$ with the following property. Whenever $\alpha,\beta,\gamma$ are $(2C,4C)$--geodesics in $X$ forming a triangle with vertices $x,y,z$:
\begin{enumerate}
\item in the former case, $\alpha,\beta,\gamma$ all come within distance $\Delta$ of the point $b$;
\item in the latter, there exist points $a,b,c\in B$ such that the quasigeodesic $\alpha$ comes within distance $\Delta$ of the points $b$ and $c$, the quasigeodesic $\beta$ does of $a$ and $c$, and finally $\gamma$ does of $a$ and $b$.
\end{enumerate}

By \Cref{lem:quasi-geodesics_in_intervals}, we can choose the above quasigeodesics so that $\alpha\sq[y,z]_{\mu}$, $\beta\sq[x,z]_{\mu}$, $\gamma\sq[x,y]_{\mu}$. If the barycentre is a point $b$, this shows that $b$ is $\Delta$--close to all three intervals $[x,y]_{\mu},[y,z]_{\mu},[z,x]_{\mu}$, and hence $b$ is (uniformly) close to $\mu(x,y,z)$, as in the proof of \Cref{cor:unique_hyp}. Since the same argument applies to $\nu$, it follows that the point $\nu(x,y,z)$ is also near $b$, and hence near $\mu(x,y,z)$.

If instead the barycentre is a subspace $B\in\mc{P}$, we have that $\mu(x,y,z)$ is within bounded distance of the point $\mu(a,b,c)$ by \Cref{lem:cm_basics}(5), and similarly $\nu(x,y,z)$ is within bounded distance of $\nu(a,b,c)$. Finally, $\mu(a,b,c)$ and $\nu(a,b,c)$ are uniformly close, because $\mu|_B$ and $\nu|_B$ are close, as observed above.

In conclusion, in all cases, we have $\mu(x,y,z)\approx_D\nu(x,y,z)$ for a constant $D\geq 0$ independent of the points $x,y,z\in X$, proving the theorem. 
\end{proof}

\section{Medians on products}\label{sec:prod_hyp}

Following \cite{Bow-hypproducts}, we say that a geodesic hyperbolic space $Z$ is \emph{bushy} if there exists a constant $\lambda\geq 0$ such that, for any point $z\in Z$, there exist three ideal points $\xi_1,\xi_2,\xi_3\in\partial_{\infty}Z$ with all pairwise Gromov products based at $x$ bounded above by $\lambda$. An $\R$--tree $T$ is \emph{furry} if removing any point of $T$ results in at least $3$ connected components. Note that all asymptotic cones of a bushy hyperbolic space are furry.

The goal of this section is to show that products of furry trees admit a unique structure of (Lipschitz) median algebra, and deduce from this that products of bushy hyperbolic spaces admit a unique coarse median structure.

\subsection{Products of trees}

The following result can be deduced by combining various results of Bowditch.

\begin{prop}\label{prop:unique_product_trees}
Let $P:=\prod_{1\leq i\leq p}\mc{T}_i$ be a finite product of furry $\R$--trees. Then $P$ admits a unique Lipschitz median operator.
\end{prop}
\begin{proof}
Let $n\colon P^3\ra P$ be the standard median operator on $P$, obtained by computing medians coordinate-wise. Let $\nu$ be some other Lipschitz median operator on $P$. Let $d$ be the metric on $P$ that is the $\ell^1$ product of the metrics on the factors $\mc{T}_i$.

Since $P$ has covering dimension equal to $p$, the median algebra $(P,\nu)$ has rank equal to $p$: one inequality follows from \cite[Theorem~2.2 and Lemma~7.6]{Bow-coarsemedian} and the other e.g.\ from \cite[Proposition~5.6]{Bow-someproperties}. By \Cref{prop:Zeidler}, it follows that there exists a metric $\rho$ on $P$ such that $(P,\rho)$ is a median space with induced median operator $\nu$, and such that $\rho$ is bi-Lipschitz equivalent to $d$. In particular, $(P,\rho)$ is a complete metric space and it has the same induced topology as $(P,d)$.

Now, \cite[Proposition~4.8]{Bow-largescale} implies that the identity map $(P,n)\ra (P,\nu)$ is an isomorphism of median algebras. In other words, $\nu=n$ as desired.
\end{proof}

\subsection{Products of hyperbolic spaces}

Let $X=\prod_{1\leq i\leq p} X_i$ be a product of bushy geodesic hyperbolic spaces. By \Cref{cor:unique_hyp}, each space $X_i$ admits a unique coarse median structure $[m_i]$. Let $m\colon X^3\ra X$ denote any coarse median obtained as the coordinate-wise product of the $m_i$. We refer to $[m]$ as the \emph{standard} coarse median structure on $X$. Our goal is now to show that this is the only possible coarse median structure on the product $X$.

Thus, consider an arbitrary coarse median $\mu$ on $X$. We wish to analyse the structure of $\mu$ by appealing to \Cref{lem:Bowditch} above. For this, we need suitable families of subsets $\mc{E},\mc{F},\mc{E}^{\infty}(\cdot),\mc{F}^{\infty}(\cdot)$ of $X$ and its asymptotic cones. The families $\mc{E}$, $\mc{E}^{\infty}(\cdot)$ and $\mc{F}^{\infty}(\cdot)$ will be defined purely in terms of the metric on $X$ and its asymptotic cones, whereas the family $\mc{F}$ will only depend on the chosen coarse median $\mu$. 

The family $\mc{E}$ consists of \emph{maximal flats} in $X$, namely products $\prod_{1\leq i\leq p}\gamma_i$ with each $\gamma_i\sq X_i$ a geodesic line. Each asymptotic cone $X^{\infty}$ of $X$ is a product of $\R$--trees $\prod_{1\leq i\leq p}T_i$. We similarly define $\mc{E}^{\infty}(X^{\infty})$ to be the family of maximal flats in $X^{\infty}$, that is, products $\prod_{1\leq i\leq p}\gamma_i$ with each $\gamma_i\sq T_i$ a geodesic line. Finally, $\mc{F}^{\infty}(X^{\infty})$ is the family of \emph{metric panels} in $X^{\infty}$: products $\prod_{1\leq i\leq p}\alpha_i$ of geodesics $\alpha_i\sq T_i$. Each $\alpha_i$ might be a single point, a geodesic segment, a ray, or a line; thus, elements of $\mc{F}^{\infty}(X^{\infty})$ need not be $p$--dimensional and they might be compact. Note that we have $\mc{E}^{\infty}(X^{\infty})\sq\mc{F}^{\infty}(X^{\infty})$. 

Before describing the family $\mc{F}$, we need a definition and a lemma.

\begin{defn}
Let $(Z,\nu)$ be a coarse median space and $h\geq 1$ a constant.
\begin{enumerate}
\item A map $f\colon M\ra Z$, where $(M,m)$ is a median algebra, is an \emph{$h$--quasimorphism} if, for all points $x,y,z\in M$, we have $\nu(f(x),f(y),f(z))\approx_hf(m(x,y,z))$.
\item If $M$ is also equipped with a metric, a map $f\colon M\ra Z$ is an \emph{$h$--strong quasimorphism} if it is both an $h$--quasimorphism and an $h$--quasi-isometric embedding.
\item An \emph{$n$--dimensional $h$--cuboid} in $(Z,\nu)$ is an $h$--quasiconvex subset of $Z$ that is also the image of an $h$--strong quasimorphism $f\colon\mf{c}\ra Z$, where $\mf{c}$ is an $n$--dimensional Euclidean cube.
\end{enumerate}
\end{defn}

Recall that $p$ is the number of factors of the product $X$. Weak and strong limits were defined in \Cref{sub:ultralimits}.

\begin{lem}\label{lem:weak_limit}
There exists a constant $h=h(X,\mu)$ such that the following holds. For every asymptotic cone $X^{\infty}$ of $X$, each element of the family $\mc{F}^{\infty}(X^{\infty})$ is a weak limit of $p$--dimensional $h$--cuboids in $X$.
\end{lem}
\begin{proof}
Recall that each asymptotic cone $X^{\infty}$ is a product $\prod_{1\leq i\leq p}T_i$ of $\R$--trees. We need to show that, for every $F\in\mc{F}^{\infty}(X^{\infty})$, every bounded subset of $F$ is contained in a strong limit of $p$--dimensional uniform-quality cuboids of $X$. Each element of $\mc{F}^{\infty}(X^{\infty})$ is exhausted by products $\prod_i\alpha_i$ of geodesic segments $\alpha_i\sq T_i$ and, up to enlarging the product, we can assume that none of the $\alpha_i$ is a single point. Thus, it suffices to consider the case when $F=\prod_{1\leq i\leq p}\alpha_i$ for non-trivial segments $\alpha_i\sq T_i$. 

The coarse median $\mu$ induces a Lipschitz median operator $\mu^{\infty}$ on $X^{\infty}$. Since the $X_i$ are bushy by hypothesis, the trees $T_i$ are furry, and thus \Cref{prop:unique_product_trees} implies that $\mu^{\infty}$ is the standard median operator on the product of trees. As a consequence, the set $F\sq X^{\infty}$ is convex for $\mu^{\infty}$, and its vertex set (i.e.\ the finite set of points all of whose coordinates are endpoints of the $\alpha_i$) is a median subalgebra of $(X^{\infty},\mu^{\infty})$ isomorphic to a $p$--dimensional discrete cube.

Let $f\colon\{0,1\}^p\ra X^{\infty}$ be an injective $0$--quasimorphism (i.e.\ a median morphism) with image the vertex set of $F$. By \cite[Lemma~11.2]{Bow-quasiflats}, there exists a constant $h'=h'(X,\mu)$ 
such that there is a sequence of $h'$--quasimorphisms $f_n\colon\{0,1\}^p\ra (X,\mu)$ converging strongly to $f$; moreover, the images of the $f_n$ are uniformly ``straight'' in Bowditch's sense.
The latter implies that the $f_n$ can be extended to $h$--strong quasimorphisms $f_n\colon\mf{c}_n\ra X$ with $h$--quasiconvex image, where the $\mf{c}_n$ are $p$--dimensional Euclidean cubes and $h=h(X,\mu)$; this follows from Lemmas~9.2 and~9.3 in \cite{Bow-quasiflats}, as well as the discussion between them.

In conclusion, the sets $f_n(\mf{c}_n)$ are $h$--cuboids in $(X,\mu)$ and they converge strongly to a subset $C^{\infty}\sq X^{\infty}$ containing the vertices of the set $F$. In the limit, the set $C^{\infty}$ is $0$--quasiconvex, that is, convex. Since $F$ is the convex hull if its vertex set (again because $\mu^{\infty}$ is standard), this shows that $F\sq C^{\infty}$. In conclusion, $F$ is a weak limit of $h$--cuboids, as we wanted.
\end{proof}

We thus define $\mc{F}$ as the collection of all $p$--dimensional $h$--cuboids in $X$, for the constant $h=h(X,\mu)$ provided by \Cref{lem:weak_limit}. We can now check that our families satisfy Bowditch's conditions.

\begin{lem}\label{lem:conditions_satisfied}
The families $\mc{E},\mc{F},\mc{E}^{\infty}(\cdot),\mc{F}^{\infty}(\cdot)$ defined above satisfy conditions (E1)--(E3) of \Cref{sub:ultralimits}.
\end{lem}
\begin{proof}
Conditions~(E2) and~(E3) are immediate from the definitions. Regarding Condition~(E1), it is also clear that all strong limits of the family $\mc{E}$ lie in $\mc{E}^{\infty}(X^{\infty})$. Strong limits of the family $\mc{F}$ lie in $\mc{F}^{\infty}(X^{\infty})$ because the median operator $\mu^{\infty}$ on $X^{\infty}$ is standard by \Cref{prop:unique_product_trees}. Finally, every set in $\mc{F}^{\infty}(X^{\infty})$ is a weak limit of sets in $\mc{F}$ by \Cref{lem:weak_limit}.
\end{proof}

We can finally deduce that maximal flats in $X$ are uniformly quasiconvex:

\begin{prop}\label{prop:quasiconvex_flats}
There exists a constant $k=k(X,\mu)$ such that all elements of $\mc{E}$ are $k$--quasiconvex.
\end{prop}
\begin{proof}
By \Cref{lem:conditions_satisfied}, we can apply \Cref{lem:Bowditch} to the families $\mc{E},\mc{F},\mc{E}^{\infty}(\cdot),\mc{F}^{\infty}(\cdot)$. As a result, there exists a constant $R=R(X,\mu)$ such that every bounded subset of an element of $\mc{E}$ is contained in the $R$--neighbourhood of an element of $\mc{F}$. 

By construction, all elements of $\mc{E}$ are bilipschitz equivalent to $\R^p$ and all elements of $\mc{F}$ embed in $\R^p$ uniformly quasi-isometrically. Moreover, the elements of $\mc{F}$ are uniformly quasiconvex. Thus, \Cref{cor:Fuerstenberg} implies that the elements of $\mc{E}$ are uniformly quasiconvex, as claimed.
\end{proof}

\begin{cor}
\label{cor:straight_cubes}
    There exists a constant $k'=k'(X,\mu)$ such that all products of $p$ geodesic segments in the $X_i$ are $k'$-quasiconvex.
\end{cor}
\begin{proof}
Since the spaces $X_i$ are bushy, there exists a constant $\lambda$ such that for any geodesic segment $\alpha_i\sq X_i$, there exist two $(\lambda,\lambda)$--quasigeodesic lines $\gamma_i,\gamma_i'\sq X_i$ with the following properties
\begin{itemize}
    \item $\alpha_i\sq\mc{N}_{\lambda}(\gamma_i)\cap \mc{N}_{\lambda}(\gamma_i')$;
    \item for any $D\geq 0$, if $x\in X_i$ satisfies $d(x,\gamma_i), d(x,\gamma'_i)\leq D$, then $d(x,\alpha_i)\leq D+\lambda$.
\end{itemize}

Now, consider a product $\prod_{1\leq i\leq p}\alpha_i$ of geodesic segments $\alpha_i\sq X_i$, and let $\gamma_i,\gamma_i'\sq X_i$ be quasigeodesic lines as above. The products $\prod_{1\leq i\leq p}\mc{N}_{\lambda}(\gamma_i)$ and $\prod_{1\leq i\leq p}\mc{N}_{\lambda}(\gamma_i')$ are $k''$--quasiconvex for some $k''=k''(X,\mu)$, by \Cref{prop:quasiconvex_flats} and \Cref{lem:cm_basics}(6).

Using the two properties above it is then easy to show that $\prod_i\alpha_i$ is uniformly quasiconvex (roughly, given $x,y\in \prod_i\alpha_i$ and $z$ in the interval between them, $z$ is close to intervals with endpoints on $\prod_{1\leq i\leq p}\gamma_i$ and $\prod_{1\leq i\leq p}\gamma'_i$ by \Cref{lem:cm_basics}(4), therefore it is close to both these two products by quasiconvexity, and therefore $z$ is close to $\prod_i\alpha_i$).
\end{proof}

\begin{thm}
\label{thm:prod_hyp_unique_median}
Let $X=\prod_{1\leq i\leq p}X_i$ for some bushy geodesic hyperbolic spaces $X_i$. Then there is a unique coarse median structure on $X$.
\end{thm}
\begin{proof}
As in the above discussion, let $[m_i]$ be the unique coarse median structures on the $X_i$, and let $[m]$ be the standard coarse median structure on $X$. Let $\mu$ an arbitrary coarse median on $X$. \Cref{cor:straight_cubes} yields a constant $k'\geq 0$ such that all products of geodesics $\prod_{1\leq i\leq p}\alpha_i$ are $k'$--quasiconvex with respect to $\mu$. In particular, for $x,y\in X$, the intervals $[x,y]_{\mu}$ are contained in uniformly bounded neighbourhoods of the intervals $[x,y]_m$ (which are at uniformly bounded Hausdorff distance from products of geodesics). An application of \Cref{lem:cm_basics}(3) then yields that $m$ and $\mu$ are close, and hence $[m]=[\mu]$ as required.
\end{proof}

\begin{rmk}\label{rmk:R_factor}
A product $Y=X\times \R$, where $X$ is an unbounded coarse median space, cannot have a unique coarse median structure. This is essentially because the product structure on $Y$ is not canonical, as we can replace the map to the $\R$--factor with any map of the form $(x,y)\mapsto y+f(x)$ for $f$ an unbounded Lipschitz function on $X$ (e.g.\ the distance function from a basepoint). Different ways of describing $Y$ as a product yield different coarse median structures.
\end{rmk}

\bibliography{./mybib}
\bibliographystyle{alpha}

\end{document}